\documentclass{amsart}

\newtheorem{theorem}[equation]{Theorem}
\newtheorem{lemma}[equation]{Lemma}
\newtheorem{corollary}[equation]{Corollary}
\newtheorem{proposition}[equation]{Proposition}

\numberwithin{equation}{section}

\usepackage{amscd,amsmath,amssymb,verbatim}

\begin{document}

\title{On the $p$-integrality of $A$-hypergeometric series} 
\author{Alan Adolphson}
\address{Department of Mathematics\\
Oklahoma State University\\
Stillwater, Oklahoma 74078}
\email{adolphs@math.okstate.edu}
\author{Steven Sperber}
\address{School of Mathematics\\
University of Minnesota\\
Minneapolis, Minnesota 55455}
\email{sperber@math.umn.edu}
\date{\today}
\keywords{}
\subjclass{Primary 33C70;  Secondary 11T23}
\begin{abstract}
Let $A$ be a set of $N$ vectors in ${\mathbb Z}^n$ and let $v$ be a vector in~${\mathbb C}^N$ 
that has minimal negative support for $A$.  Such a vector $v$ gives rise to a formal series 
solution of the $A$-hypergeometric system with parameter $\beta=Av$.  If $v$ lies 
in~${\mathbb Q}^n$, then this series has rational coefficients.   Let $p$ be a prime number.  
We characterize those $v$ whose coordinates are rational, $p$-integral, and lie in the closed 
interval $[-1,0]$ for which the corresponding normalized series solution has $p$-integral 
coefficients.  
\end{abstract}
\maketitle

\section{Introduction}

The arithmetic nature of coefficients of series solutions of differential equations is basic in a number of areas including transcendence theory and arithmetic geometry.  In the case of classical hypergeometric functions, Dwork in \cite{D} and \cite{D2} gave conditions for the $p$-integrality of the coefficients (and hence the $p$-adic boundedness of the series) as well as conditions for continuation of certain ratios of $p$-adic analytic functions.  In the present work we consider the important examples of $A$-hypergeometric functions with rational character values and determine conditions for $p$-integrality of the coefficients.  This leads naturally to conditions for rational integrality of the coefficients as well.  In particular we extend some previous work to the multivariable case.  

Questions of integrality of these coefficients arise in arithmetic mirror symmetry and in arithmetic algebraic geometry.  In a subsequent work we produce unit root formulas (and others) for families of projective hypersurfaces, where such integrality conditions prove useful.

Let $A=\{ {\bf a}_1,\dots,{\bf a}_N \}\subseteq{\mathbb Z}^n$ and let $L\subseteq{\mathbb Z}^N$ 
be the lattice of relations on~$A$:
\[ L = \bigg\{l=(l_1,\dots,l_N)\in{\mathbb Z}^N\;\bigg|\; \sum_{i=1}^N l_i{\bf a}_i = {\bf 0}
\bigg\}. \]
Let $\beta = (\beta_1,\dots,\beta_n)\in{\mathbb C}^n$.  The {\it $A$-hypergeometric system with 
parameter $\beta$\/} is the system of partial differential operators in $\lambda_1,\dots,
\lambda_N$ consisting of the {\it box operators\/}
\begin{equation}
\Box_l = \prod_{l_i>0} \bigg( \frac{\partial}{\partial \lambda_i}\bigg)^{l_i} - \prod_{l_i<0} \bigg( 
\frac{\partial}{\partial \lambda_i}\bigg)^{-l_i} \quad\text{for $l\in L$}
\end{equation}
and the {\it Euler\/} or {\it homogeneity\/} operators
\begin{equation}
Z_i = \sum_{j=1}^N a_{ij}\lambda_j\frac{\partial}{\partial\lambda_j} -\beta_i\quad\text{for $i=1,
\dots,n$},
\end{equation}
where ${\bf a}_j = (a_{1j},\dots,a_{nj})$.  If there is a linear form $h$ on ${\mathbb R}^n$ such 
that $h({\bf a}_i)=1$ for $i=1,\dots,N$, we call this system {\it nonconfluent}; otherwise, we 
call it {\it confluent}.

Let $v=(v_1,\dots,v_N)\in{\mathbb C}^N$.  The {\it negative support\/} of $v$ is the set
\[ {\rm nsupp}(v) = \{i\in\{1,\dots,N\} | \text{ $v_i$ is a negative integer}\}.  \]
Let
\[ L_v = \{l\in L\mid {\rm nsupp}(v+l) = {\rm nsupp}(v)\} \]
and put
\begin{equation}
\phi_v(\lambda) = \sum_{l\in L_v} \frac{[v]_{l_-}}{[v+l]_{l_+}} \lambda^{v+l},
\end{equation}
where
\[ [v]_{l_-} = \prod_{l_i<0} \prod_{j=1}^{-l_i} (v_i-j+1) \]
and
\[ [v+l]_{l_+} = \prod_{l_i>0}\prod_{j=1}^{l_i} (v_i+j). \]
The vector $v$ is said to have {\it minimal negative support\/} if there is no $l\in L$ for 
which ${\rm nsupp}(v+l)$ is a proper subset of ${\rm nsupp}(v)$.  The series $\phi_v(\lambda)$ 
is a formal solution of the system (1.1), (1.2) for $\beta=\sum_{i=1}^N v_i{\bf a}_i$ if and only 
if $v$ has minimal negative support (see Saito-Sturmfels-Takayama\cite[Proposition~3.4.13]{SST}).

Let $p$ be a prime number.  In \cite{D2,D1}, Dwork introduced the idea of normalizing 
hypergeometric series to have $p$-adic radius of convergence equal to~$1$.  This involves 
simply replacing each variable $\lambda_i$ by $\pi\lambda_i$, where $\pi$ is any solution of 
$\pi^{p-1} = -p$.  We define the $p$-adically normalized hypergeometric series to be
\begin{equation}
\Phi_v(\lambda) = \sum_{l\in L_v} \frac{[v]_{l_-}}{[v+l]_{l_+}} \pi^{\sum_{i=1}^N l_i}\lambda^{v+l} 
\quad\big(=\pi^{-v}\phi_v(\pi\lambda)\big).
\end{equation}
Note that for nonconfluent $A$-hypergeometric systems one has $\sum_{i=1}^N l_i = 0$, so in that 
case we have $\Phi_v(\lambda) = \phi_v(\lambda)$, i.e., the normalized series is just the usual 
one.  In this paper we study the $p$-integrality of the coefficients of the $\lambda^{v+l}$ in 
$\Phi_v(\lambda)$.

Let ${\mathbb N}$ denote the set of nonnegative integers and ${\mathbb N}_+$ the set of positive integers.  Every $t\in{\mathbb N}$ has a $p$-adic expansion
\[ t=t_0+t_1p+\cdots+t_{b-1}p^{b-1}, \quad\text{$0\leq t_j\leq p-1$ for all $j$.} \]
We define the {\it $p$-weight\/} of $t$ to be $w_p(t) = \sum_{j=0}^{b-1} t_j$.  For later use, it will be convenient to define $t!! = \prod_{j=0}^{b-1} t_j!$.  
These definitions are extended to vectors of nonnegative integers componentwise:
if $s=(s_1,\dots,s_N)\in{\mathbb N}^N$, define $w_p(s) = \sum_{i=1}^N w_p(s_i)$ and $s!! = \prod_{i=1}^N s_i!!$.

Let $R$ be the set of all $p$-integral rational vectors $(r_1,\dots,r_N)\in({\mathbb Q}\cap{
\mathbb Z}_p)^N$ satisfying $-1\leq r_i\leq 0$ for $i=1,\dots,N$.  For $r\in R$ choose a power 
$p^a$ such that $(1-p^a)r\in{\mathbb N}^N$ and set $s=(1-p^a)r$.  We define a {\it 
weight function\/} $w$ on $R$ by setting $w(r) = w_p(s)/a$.  The positive integer $a$ is not 
uniquely determined by~$r$ but the ratio $w_p(s)/a$ is independent of the choice of $a$ and 
depends only on $r$.  Note that since $0\leq (1-p^a)r_i\leq p^a-1$ for all $i$ we have $0\leq 
w_p(s)\leq aN(p-1)$ and $0\leq w(r)\leq N(p-1)$.  

We consider $A$-hypergeometric systems (1.1), (1.2) for those $\beta$ for which the set
\[ R_\beta = \bigg\{r=(r_1,\dots,r_N)\in R\;\bigg|\; \sum_{i=1}^N r_i{\bf a}_i = \beta\bigg\} \]
is nonempty.  Define $w(R_\beta) = \inf\{w(r)\mid r\in R_\beta\}$.  Trivially, if $v\in R_\beta$ 
then $w(v)\geq w(R_\beta)$.
Our first main result is the following statement.
\begin{theorem}
If $v\in R_\beta$, then the series $\Phi_v(\lambda)$ has $p$-integral coefficients if and only if
$w(v)=w(R_\beta)$.  If $w(v)>w(R_\beta)$, then the coefficients of $\Phi_v(\lambda)$ are 
$p$-adically unbounded.
\end{theorem}

Note that we do not assume in Theorem 1.5 that $v$ has minimal negative support, so the series 
$\Phi_v(\lambda)$ is not necessarily a solution of the system (1.1),~(1.2).

Direct application of Theorem 1.5 is limited by the fact that we do not know a general procedure for computing $w(R_\beta)$ or for determining whether there exists $v\in R_\beta$ such that $w(v) = w(R_\beta)$.  Our second main result (Theorem 3.5 below) gives a lower bound for $w(R_\beta)$.  In many cases of interest, one can find $v\in R_\beta$ for which $w(v)$ equals this lower bound.  For such $v$ we have $w(v)=w(R_\beta)$, hence $\Phi_v(\lambda)$ has $p$-integral coefficients.  As an application, in Section~4 we prove 
Proposition~4.7, a $p$-integrality criterion for ${}_rF_{s-1}$-hypergeometric series ($s\geq r$) 
with parameters in the interval $(0,1]$.  In many cases, including the nonconfluent case, 
Proposition~4.7 improves the $p$-integrality criterion of Dwork\cite[Lemma~2.2]{D2}.  In 
Section 5, we restrict to the nonconfluent case and find conditions that guarantee the series 
$\Phi_v(\lambda)$ has integer coefficients.  This situation often occurs for families of 
complete intersections.

The problem of computing $w(R_\beta)$ becomes more 
tractable when one considers the set of parameters $\beta+{\mathbb Z}^n$ simultaneously.  
Applying ideas of R. Blache\cite[Section~1.1]{B}, we show in Section~6 that there exist 
$\beta_0\in\beta+{\mathbb Z}^n$ and $v\in R_{\beta_0}$ such that
\[ w(v) = w(R_{\beta_0}) = \min\{w(R_{\gamma})\mid \gamma\in\beta+{\mathbb Z}^n\}. \]
(The main result, Theorem 6.1, is in fact stronger than this assertion.)  
In particular, the series $\Phi_v(\lambda)$ has $p$-integral coefficients.  The proof gives an 
effective procedure for finding $\beta_0$ and $v$.  We remark that our definition of 
$w(R_\beta)$ is closely related to Blache's notion\cite{B} of $p$-density for finite subsets of~${\mathbb N}^n$.  

The ideas in this article arose in the study of exponential sums over finite fields.  In Section 7 we apply the results of Section 6 to show that every exponential sum on the torus over a finite field of characteristic~$p$ gives rise to a nonempty family of $p$-integral $A$-hypergeometric series.  These series are worthy of further study as they seem to carry useful information about the exponential sum.  One such example is studied in \cite{AS1}.

We regard the results of this paper as foundational in nature.  In \cite{D2}, Dwork proved not only $p$-integrality results but showed under additional hypotheses that certain ratios of hypergeometric series have $p$-adic analytic continuation and that the series $q(t):=\exp\big(G(t)/F(t)\big)$ is $p$-integral for log solutions $F(t)\log t+ G(t)$ of certain hypergeometric equations.  As is well-known, this series $q(t)$ appears in works on mirror symmetry in connection with the mirror map.  One of our goals is to prove such results for $A$-hypergeometric series and to relate these results to properties of the corresponding exponential sums.  A related goal is to prove $p$-adic analytic formulas for unit roots.  An instance of this was given in \cite{AS1}, where we study the $L$-function of a general family of toric exponential sums, prove $p$-adic continuation for a ratio of related $p$-integral $A$-hypergeometric series, and identify the unique unit root of this $L$-function as a special value of this continued function.  In a recent as yet unpublished work, we consider families of projective Calabi-Yau hypersurfaces and give a $p$-adic analytic formula for the unique unit root of the middle-dimensional factor of the zeta function in similar such terms.  An early version of this result appears in \cite{AS2}.  In these cases, the required $p$-integrality of the relevant $A$-hypergeometric series is more obvious.

\section{Proof of Theorem 1.5}

We begin with some elementary results.  Let $t\in{\mathbb N}$ and define 
\[ \alpha_p(t) = \frac{t-w_p(t)}{p-1}\in{\mathbb N}. \]
One has the easily checked congruence
\[ t!\equiv (-p)^{\alpha_p(t)}t!!\pmod{p^{\alpha_p(t) + 1}}. \]
For our purposes, this is more conveniently written in the form
\begin{equation} (-p)^{-\alpha_p(t)}t!\equiv t!!\pmod{p}, 
\end{equation}
where both sides of the congruence are $p$-adic units.

Now let $k\in{\mathbb N}$ with $k\leq t$.  Since
\[ t(t-1)\cdots (t-k+1) = \frac{t!}{(t-k)!}, \]
Eq.~(2.1) gives us a congruence for this product.  If we set
\[ \beta_p(t,k) = \alpha_p(t) - \alpha_p(t-k) = \frac{k-w_p(t) + w_p(t-k)}{p-1}\in{\mathbb N}, \]
then
\begin{equation}
(-p)^{-\beta_p(t,k)}\prod_{i=0}^{k-1} (t-i) \equiv \frac{t!!}{(t-k)!!}\pmod{p} 
\end{equation}
and both sides of this congruence are $p$-adic units.

We extend the congruence (2.2) to $t\in{\mathbb Z}_p$.  Write
\[ t = \sum_{i=0}^\infty t_ip^i,\quad \text{$0\leq t_i\leq p-1$ for all $i$,} \]
and set $t^{(b)} = \sum_{i=0}^{b-1} t_ip^i\in{\mathbb Z}_{\geq 0}$.  Then $t\equiv t^{(b)}\pmod{p^b}$.  

\begin{lemma}
Let $k\in{\mathbb N}$ with $k\leq t^{(b)}$.  One has
\[ \prod_{i=0}^{k-1} (t-i) \equiv \prod_{i=0}^{k-1} (t^{(b)}-i)\pmod{p^{\beta_p(t^{(b)},k)+1}}. \]
\end{lemma}

\begin{proof}
Write $t=t^{(b)}+p^b\epsilon$, where $\epsilon\in{\mathbb Z}_p$.  We have
\[ \prod_{i=0}^{k-1} (t-i) =\prod_{i=0}^{k-1} (t^{(b)}-i) + \sum_{j=1}^k \sum_{0\leq i_1<\dots<i_j\leq k-1} M(i_1,\dots,i_j), \]
where
\[ M(i_1,\dots,i_j) = t^{(b)}(t^{(b)}-1)\cdots \widehat{(t^{(b)}-i_1)}\cdots \widehat{(t^{(b)}-i_j)}\cdots (t^{(b)}-k+1) (p^b\epsilon)^j. \]
We have $0<t^{(b)}-i_j<\dots<t^{(b)}-i_1\leq p^b-1$, so each of these positive integers has $p$-ordinal $<b$.  This implies that
\[ {\rm ord}_p\:M(i_1,\dots,i_j)> {\rm ord}_p\:\prod_{i=0}^{k-1} (t^{(b)}-i) = \beta_p(t^{(b)},k), \]
where the last equality follows from (2.2).  
\end{proof}

It follows from (2.2) and Lemma 2.3 that
\begin{equation}
(-p)^{-\beta_p(t^{(b)},k)}\prod_{i=0}^{k-1} (t-i) \equiv \frac{t^{(b)}!!}{(t^{(b)}-k)!!} \pmod{p}
\end{equation}
and that both sides of this congruence are $p$-adic units.  We shall use (2.4) to determine the $p$-divisibility of the coefficients of the series~(1.4).  For $v=(v_1,\dots,v_N)\in{\mathbb Z}_p^N$, we write the $p$-adic expansions of the $v_i$ as
\[ v_i = \sum_{j=0}^\infty v_{ij}p^j,\quad\text{$0\leq v_{ij}\leq p-1$ for all $j$,} \]
and we set
\[ v_i^{(b)} = \sum_{j=0}^{b-1} v_{ij}p^j. \]
We put $v^{(b)} = (v_1^{(b)},\dots,v_N^{(b)})\in \{0,1,\dots,p^b-1\}^N$.  

\begin{proposition}
Let $v\in{\mathbb Z}_p^N$ and $l\in{\mathbb Z}^N$ with ${\rm nsupp}(v+l) = {\rm nsupp}(v)$.  For all sufficiently large positive integers $b$, one has 
\begin{equation}
0\leq v_i^{(b)} + l_i\leq p^b-1\quad\text{for $i=1,\dots,N$.}
\end{equation}
When $(2.6)$ holds, we have
\begin{equation}
\pi^{w_p(v^{(b)})-w_p((v+l)^{(b)})}\frac{\pi^{\sum_{i=1}^N l_i}[v]_{l_-}}{[v+l]_{l_+}} \equiv 
 \frac{v^{(b)}!!}{(v+l)^{(b)}!!}\pmod{p}.
\end{equation}
\end{proposition}

Note that (2.6) implies that $(v+l)^{(b)} = v^{(b)}+l$.

Since both sides of the congruence (2.7) are $p$-adic units, we get the following corollary.
\begin{corollary}
With the hypotheses  and notation of Proposition $2.5$, if $(2.6)$ holds for a positive integer $b$, then
\begin{equation}
{\rm ord}_p\: \frac{\pi^{\sum_{i=1}^N l_i}[v]_{l_-}}{[v+l]_{l_+}} = \frac{1}{p-1} \big(w_p((v+l)^{(b)})-w_p(v^{(b)})\big).
\end{equation}
\end{corollary}

\begin{proof}[Proof of Proposition $2.5$]
The congruence (2.7) follows directly from (2.6), (2.4), and the definitions of $[v]_{l_-}$, $[v+l]_{l_+}$, and $\beta_p(v_i^{(b)},l_i)$, so it remains to prove only~(2.6).  There are three cases to consider.

Suppose first that $v_{ij}=0$ for all sufficiently large $j$.  Then $v_i$ is a nonnegative integer and for all large $b$ we have $v_i = v_i^{(b)}$.  If $l_i<0$, then $v_i +l_i\geq 0$ since ${\rm nsupp}(v+l) = {\rm nsupp}(v)$.  If $l_i>0$, we can guarantee that $v_i+l_i\leq p^b-1$ simply by choosing $b$ to be sufficiently large.

Suppose next that $v_{ij} = p-1$ for all sufficiently large $j$.  In this case we have for all sufficiently large $b$
\[ v_i = \sum_{j=0}^{b-1} v_{ij}p^j + p^b\sum_{j=0}^\infty (p-1)p^j = \sum_{j=0}^{b-1} v_{ij}p^j -p^b, \]
i.e., $v_i$ is a negative integer.  If $l_i<0$, we will have $v_i^{(b)}+l_i\geq 0$ for $b$ sufficiently large: since $v_{ij}=p-1$ for all sufficiently large $j$, we can make $v_i^{(b)}$ arbitrarily large by choosing $b$ large.  If $l_i>0$, then $v_i+l_i<0$ since ${\rm nsupp}(v+l) = {\rm nsupp}(v)$.  But $v_i+l_i = v_i^{(b)}-p^b+l_i$, so $v_i^{(b)}+l_i<p^b$.  

Finally, suppose that there are infinitely many $j$ for which $0<v_{ij}<p-1$.  Suppose that $l_i<0$.  Since $v_{ij}>0$ for infinitely many $j$, we can make $v_i^{(b)}$ arbitrarily large by choosing $b$ large and thus guarantee that $v_i^{(b)}+l_l\geq 0$.  And since $v_{ij}<p-1$ for infinitely many $j$, there are infinitely many $b$ such that
\[ v_i^{(b)}\leq \sum_{j=0}^{b-2}(p-1)p^j + (p-2)p^{b-1} = p^b-1-p^{b-1}. \]
By choosing such a $b$ large enough we can guarantee that $v_i^{(b)}+l_i\leq p^b-1$ when $l_i>0$.  But when this inequality holds for one $b$, it automatically holds for all larger $b$ as well. 
\end{proof}

When $v\in R$, one has more control over the right-hand side of (2.9).  If $b\in{\mathbb N}$ is chosen so that $(1-p^b)v\in{\mathbb N}^N$, then $(1-p^b)v=v^{(b)}$, so
\[ w_p(v^{(b)}) = b w(v). \]
If $b$ is sufficiently large and $l\in{\mathbb Z}^N$ satisfies ${\rm nsupp}(v+l) = {\rm nsupp}(v)$, then (2.6) implies that $v+(1-p^b)^{-1}l\in R$ as well, so $(1-p^b)v + l = v^{(b)}+l = (v+l)^{(b)}$ and
\[ w_p((v+l)^{(b)}) = bw(v+(1-p^b)^{-1}l). \]
For $v\in R$, Corollary 2.8 becomes the following.
\begin{corollary}
Let $v\in R$ and let $l\in{\mathbb Z}^N$ with ${\rm nsupp}(v+l) = {\rm nsupp}(v)$.  Then for all sufficiently large $b\in{\mathbb N}_+$ satisfying $(1-p^b)v\in{\mathbb N}^N$ we have $v+(1-p^b)^{-1}l\in R$.  In this case 
\begin{equation}
{\rm ord}_p\: \frac{\pi^{\sum_{i=1}^N l_i}[v]_{l_-}}{[v+l]_{l_+}} = \frac{b}{p-1} \big(w(v+(1-p^b)^{-1}l)-w(v)\big).
\end{equation}
\end{corollary}

To complete the proof of Theorem 1.5 we need a lemma.
\begin{lemma}
Let $v\in R_\beta$.  Then $r\in R_\beta$ if and only if $r=v+(1-p^b)^{-1}l$ for some $b\in{\mathbb N}_+$ satisfying $(1-p^b)v\in{\mathbb N}^N$ and some $l\in L_v$.
\end{lemma}

\begin{proof}
If $l\in L_v$, then Corollary 2.10 implies that $v+(1-p^b)^{-1}l\in R_\beta$ for sufficiently large $b$ satisfying $(1-p^b)v\in{\mathbb N}^N$.  Conversely, let $r\in R_\beta$ and let $b\in{\mathbb N}_+$ be chosen so that $(1-p^b)v,(1-p^b)r\in{\mathbb N}^N$.  Put $l=(1-p^b)(r-v)$, so that $r=v+(1-p^b)^{-1}l$.  We claim that $l\in L_v$.  Since $-1\leq v_i\leq 0$ for all $i$, to show that ${\rm nsupp}(v+l) = {\rm nsupp}(v)$ we need consider only the cases $v_i=-1$ and $v_i=0$.  Suppose that $v_i=-1$ for some $i$.  Then $0\leq r_i-v_i\leq 1$ so $1-p^b\leq l_i\leq 0$ and $v_i+l_i$ is a negative integer.  If $v_i=0$ for some $i$, then $-1\leq r_i-v_i\leq 0$, so $0\leq l_i\leq p^b-1$ and $v_i+l_i\in{\mathbb N}$.
\end{proof}

\begin{proof}[Proof of Theorem $1.5$]
Fix $v\in R_\beta$.  By Lemma 2.12
\begin{multline*}
 R_\beta = \{v+(1-p^b)^{-1}l\mid \\ 
\text{$b\in{\mathbb N}_+$, $(1-p^b)v\in{\mathbb N}^N$, $l\in L_v$, and $v+(1-p^b)^{-1}l\in R$}\}. 
\end{multline*}
The first sentence of Theorem~1.5 then follows from Corollary~2.10.  If $w(v)>w(R_\beta)$, then there exists $r\in R_\beta$ with $w(r)<w(v)$.  Choose $b\in{\mathbb N}_+$ such that $(1-p^b)v,(1-p^b)r\in{\mathbb N}^N$.  For each positive integer $c$, define $l^{(c)} = (1-p^{bc})(r-v)$.  The proof of Lemma~2.12 shows that $l^{(c)}\in L_v$.  Corollary~2.10 gives
\begin{align*}
{\rm ord}_p\: \frac{\pi^{\sum_{i=1}^N l^{(c)}_i}[v]_{l^{(c)}_-}}{[v+l^{(c)}]_{l^{(c)}_+}} &= \frac{bc}{p-1} \big(w(v+(1-p^{bc})^{-1}l^{(c)})-w(v)\big) \\
 &=\frac{bc}{p-1}\big(w(r)-w(v)\big).
\end{align*}
Since $w(r)-w(v)<0$ and $c$ is an arbitrary positive integer, this implies the second sentence of Theorem~1.5.
\end{proof}

We collect here a few results on truncations of the series $\Phi_v(\lambda)$ that will be used in Section 7.
For $b\in{\mathbb N}_+$, define
\[ R_{\beta,b} = \{r\in R_\beta \mid (1-p^b)r\in{\mathbb N}^N\}, \]
a finite set.  Put $w(R_{\beta,b}) = \min\{w(r)\mid r\in R_{\beta,b}\}$.  
For $v\in R_{\beta,b}$, let
\[ L_{v,b} = \{l\in L_v\mid v+(1-p^b)^{-1}l\in R_{\beta,b}\}. \]
The proof of Lemma 2.12 shows that the maps $r\mapsto (1-p^b)(r-v)$, $l\mapsto v+(1-p^b)^{-1}l$ define a one-to-one correspondence between $R_{\beta,b}$ and $L_{v,b}$.  Define the truncated series $\Phi_{v,b}(\lambda)$ by
\[ \Phi_{v,b}(\lambda) = \sum_{l\in L_{v,b}} \frac{[v]_{l_-}}{[v+l]_{l_+}} \pi^{\sum_{i=1}^N l_i}\lambda^{v+l}. \]

The argument that proved Theorem 1.5 also establishes the following result.
\begin{proposition}
Let $v\in R_{\beta,b}$.  Then $\Phi_{v,b}(\lambda)$ has p-integral coefficients if and only if $w(v) = w(R_{\beta,b})$.
\end{proposition}

If $v\in R_{\beta,b}$ satisfies $w(v) = w(R_{\beta,b})$, then (2.7) implies
\begin{equation}
\Phi_{v,b}(\lambda) \equiv \sum_{l\in L_{v,b}} \pi^{w_p((v+l)^{(b)})-w_p(v^{(b)})}\frac{\big((1-p^b)v\big)!!}{\big((1-p^b)v+l\big)!!}\lambda^{v+l}\pmod{p}.
\end{equation}
From this congruence it follows that
\begin{multline}
\frac{\pi^{w_p(v^{(b)})}\lambda^{-p^bv}}{\big((1-p^b)v\big)!!} \Phi_{v,b}(\lambda) \equiv \\ 
\sum_{l\in L_{v,b}} \pi^{w_p((v+l)^{(b)})}\frac{\lambda^{(1-p^b)v+l}}{\big( (1-p^b)v+l\big)!!} \pmod{p\pi^{w_p(v^{(b)})}}.
\end{multline}
Since the map $l\mapsto v+(1-p^b)^{-1}l$ is a bijection between $L_{v,b}$ and $R_{\beta,b}$ and since $(v+l)^{(b)} = (1-p^b)v+l$, we can rewrite the right-hand side of (2.15):
\begin{multline}
\frac{\pi^{w_p(v^{(b)})}\lambda^{-p^bv}}{\big((1-p^b)v\big)!!} \Phi_{v,b}(\lambda) \equiv \\
\sum_{r\in R_{\beta,b}} \pi^{w_p((1-p^b)r)} \frac{\lambda^{(1-p^b)r}}{\big( (1-p^b)r\big)!!} \pmod{p\pi^{bw(R_{\beta,b})}},
\end{multline}
an expression which is independent of the choice of $v$.

We derive from (2.16) a congruence that will be used in Section 7.  If $r\in R_{\beta,b}$ satisfies $w(r)>w(R_{\beta,b})$, then $w_p\big((1-p^b)r\big)>bw(R_{\beta,b})$, so
\begin{equation}
\frac{\lambda^{-p^bv}}{\big((1-p^b)v\big)!!} \Phi_{v,b}(\lambda) \equiv \sum_{\substack{r\in R_{\beta,b}\\ w(r) = w(R_{\beta,b})}}  \frac{\lambda^{(1-p^b)r}}{\big( (1-p^b)r\big)!!} \pmod{\pi}.
\end{equation}

\section{Lower bound for $w(R_\beta)$}

In this section we give a lower bound for $w(R_\beta)$, which leads to a useful condition for 
the $p$-integrality of $\Phi_v(\lambda)$.

Let $\Delta$ be the convex hull of the set $A\cup\{{\bf 0}\}$ and let $C(\Delta)$ be the real 
cone generated by $\Delta$.  For $\gamma\in C(\Delta)$, let $w_\Delta(\gamma)$ be the smallest 
nonnegative real number $w$ such that $w\Delta$ (the dilation of $\Delta$ by the factor $w$) 
contains $\gamma$.  It is easily seen that
\begin{equation}
w_\Delta(\gamma) = \min\bigg\{\sum_{i=1}^N t_i\:\bigg|\: (t_1,\dots,t_N)\in
({\mathbb R}_{\geq 0})^N \text{ and } \sum_{i=1}^N t_i{\bf a}_i =\gamma \bigg\}. 
\end{equation}
If $\gamma\in{\mathbb Q}^n\cap C(\Delta)$, then we may replace $({\mathbb R}_{\geq 0})^N$ by 
$({\mathbb Q}_{\geq 0})^N$ in (3.1).  For any subset $X\subseteq 
C(\Delta)$, define $w_\Delta(X) = \inf\{w(\gamma)\mid\gamma\in X\}$.

For a nonnegative integer $t$, written $t=t_0+t_1p+\cdots+t_{a-1}p^{a-1}$ with $0\leq t_i\leq p-1$ 
for $i=0,\dots,a-1$, set
\[ t' = t_1+t_2p+\cdots+t_{a-1}p^{a-2}+t_0p^{a-1}. \]
We denote by $t^{(k)}$ the $k$-fold iteration of this operation.  We extend this definition to 
vectors of nonnegative integers componentwise:  if $s=(s_1,\dots,s_N)\in{\mathbb N}^N$, put $s^{(k)} = (s_1^{(k)},\dots,s_N^{(k)})$.
For $r\in R$, choose a power $p^a$ for which $(1-p^a)r\in{\mathbb N}^N$ and set 
$s=(1-p^a)r$.  We define
$r^{(k)} = (1-p^a)^{-1}s^{(k)}\in R$.  Note that the $p$-weight of $t$ satisfies the formula 
\[ w_p(t) = \frac{p-1}{p^a-1}\sum_{k=0}^{a-1} t^{(k)}. \]
It follows that $w_p(s)=\frac{p-1}{p^a-1}\sum_{k=0}^{a-1}\sum_{i=1}^N s_i^{(k)}$, hence
\begin{equation}
w(r) = -\frac{p-1}{a} \sum_{k=0}^{a-1} \sum_{i=1}^N r_i^{(k)}. 
\end{equation}

Let $\sigma_{-\beta}$ be the smallest closed face of $C(\Delta)$ that contains $-\beta$, let $T$ be the 
set of proper closed faces of $\sigma_{-\beta}$, and put
\[ \sigma^\circ_{-\beta} = \sigma_{-\beta}\setminus\bigcup_{\tau\in T} \tau. \]
By the minimality of $\sigma_{-\beta}$, $-\beta\not\in\tau$ for any $\tau\in T$, so $-\beta\in
\sigma_{-\beta}^\circ$.  
\begin{lemma}
For all $r\in R_\beta$ and all $k$, $-\sum_{i=1}^N r_i^{(k)}{\bf a}_i\in\sigma^\circ_{-\beta}$.
\end{lemma}

\begin{proof}
Since $-\beta\in\sigma_{-\beta}$ and $-\sum_{i=1}^N r_i{\bf a}_i = -\beta$ we must have $r_i=0$ for 
all $i$ such that ${\bf a}_i\not\in \sigma_{-\beta}$.  But $r_i=0$ implies $r^{(k)}_i=0$, so 
$-\sum_{i=1}^N r_i^{(k)}{\bf a}_i\in\sigma_{-\beta}$.  Fix $\tau\in T$.  Since $-\beta\not\in\tau$, we have $r_i\neq 0$ for some $i$ such that 
${\bf a}_i\not\in\tau$.  But $r_i\neq 0$ implies that $r_i^{(k)}\neq 0$, so $-\sum_{i=1}^N 
r_i^{(k)}{\bf a}_i\not\in\tau$.
\end{proof}

If $r,\tilde{r}\in R_\beta$, then $-\sum_{i=1}^N r_i^{(k)}{\bf a}_i\equiv -\sum_{i=1}^N 
\tilde{r}_i^{(k)}{\bf a}_i\pmod{{\mathbb Z}^n}$.  Choose $\beta^{(k)}\in{\mathbb Q}^n$ such that 
$-\sum_{i=1}^N r_i^{(k)}{\bf a}_i\in -\beta^{(k)}+{\mathbb Z}^n$ for all $r\in R_\beta$.  
It now follows from~(3.1) and Lemma 3.3 that for all $r\in R_\beta$,
\[ -\sum_{i=1}^N r_i^{(k)}\geq w_\Delta\bigg(-\sum_{i=1}^N r_i^{(k)}{\bf a}_i\bigg)\geq w_\Delta\big(
\sigma_{-\beta}^\circ\cap(-\beta^{(k)}+{\mathbb Z}^n)\big). \]
Equation (3.2) therefore gives
\begin{equation}
w(r)\geq \frac{p-1}{a}\sum_{k=0}^{a-1} w_\Delta\big(\sigma_{-\beta}^\circ\cap(-\beta^{(k)}+{\mathbb Z}^n)\big)
\end{equation}
for all $r\in R_\beta$.  

Note that the right-hand side of (3.4) is independent of the choice of $a$: if~$e$ is the smallest positive 
integer such that $-\beta^{(e)}+{\mathbb Z}^n = -\beta+{\mathbb Z}^n$, then $e$ divides $a$ and 
one has
\[ \sum_{k=0}^{a-1} w_\Delta\big(\sigma_{-\beta}^\circ\cap(-\beta^{(k)}+{\mathbb Z}^n)\big) = \frac{a}{e} 
\sum_{k=0}^{e-1} w_\Delta\big(\sigma_{-\beta}^\circ\cap(-\beta^{(k)}+{\mathbb Z}^n)\big), \]
so (3.4) becomes
\[ w(r)\geq \frac{p-1}{e}\sum_{k=0}^{e-1} w_\Delta\big(\sigma_{-\beta}^\circ\cap(-\beta^{(k)}+{
\mathbb Z}^n)\big). \]
Since this estimate holds for all $r\in R_\beta$, we have established the following result.
\begin{theorem}
With the above notation, we have
\[ w(R_\beta)\geq \frac{p-1}{e}\sum_{k=0}^{e-1} w_\Delta\big(\sigma_{-\beta}^\circ\cap(-\beta^{(k)}+{
\mathbb Z}^n)\big). \]
\end{theorem}

The following corollary is now an immediate consequence of Theorem 1.5.
\begin{corollary}
If $v\in R_\beta$ satisfies $w(v) = \frac{p-1}{e}\sum_{k=0}^{e-1} w_\Delta\big(\sigma_{-\beta}^\circ\cap
(-\beta^{(k)}+{\mathbb Z}^n)\big)$, then the series $\Phi_v(\lambda)$ has $p$-integral coefficients.
\end{corollary}

{\bf Example 1:} Let $A=\{{\bf a}_1,{\bf a}_2,{\bf a}_3\}\subseteq{\mathbb Z}^2$, where 
${\bf a}_1 = (3,0)$, ${\bf a}_2 = (0,3)$, ${\bf a}_3 = (2,2)$; take $\beta = (-1,-1)$ and 
$v=(0,0,-1/2)$.  One has 
\[ L=\{(2l,2l,-3l)\mid l\in{\mathbb Z}\}, \quad L_v = \{(2l,2l,-3l)\mid l\in{\mathbb N}\}, \]
and
\[ \Phi_v(\lambda) = \lambda_3^{-1/2}\sum_{l=0}^\infty \frac{(-1)^{3l}(1/2)_{3l}}{(2l)!^2} \pi^l
\bigg(\frac{\lambda_1^2\lambda_2^2}{\lambda_3^3}\bigg)^l, \]
where we have used the Pochhammer notation: $(\alpha)_l = \alpha(\alpha+1)\cdots(\alpha+l-1)$.  
For an odd prime $p$ the vector $v$ is $p$-integral and $w(v) = (p-1)/2$.  The polytope 
$\Delta$ is the quadrilateral with vertices at $(0,0)$, $(3,0)$, $(0,3)$, and $(2,2)$, 
$C(\Delta) = ({\mathbb R}_{\geq 0})^2$, $\sigma_{-\beta}^\circ = ({\mathbb R}_{>0})^2$, and 
$-\beta+{\mathbb Z}^2 = {\mathbb Z}^2$.  We have $e=1$ and $w_\Delta\big(\sigma_{-\beta}^\circ\cap
(-\beta+{\mathbb Z}^2)\big) = 1/2$.  Corollary~3.6 implies that $\Phi_v(\lambda)$ has $p$-integral 
coefficients.

\section{Classical hypergeometric series}

In this section we deduce from Corollary 3.6 a criterion for the $p$-integrality of the ($p$-adically normalized) classical hypergeometric series
\begin{equation}
{}_rF_{s-1}\bigg(\begin{matrix} \theta_1 & \theta_2 & \dots & \theta_r \\  \sigma_1 & \sigma_2 & 
\dots & \sigma_{s-1} \end{matrix} \bigg| \pi^{s-r}t\bigg) = \sum_{j=0}\frac{(\theta_1)_j\cdots(
\theta_r)_j\pi^{(s-r)j}}{j!(\sigma_1)_j\cdots(\sigma_{s-1})_j}t^j
\end{equation}
with $s\geq r$.  In some cases, this improves on Dwork's criterion \cite[Lemma~2.2]{D2} (see 
the remark following Proposition~4.10).  Dwork assumes that the $\theta_i$ and $\sigma_j$ are 
$p$-integral rational numbers not lying in $-{\mathbb N}$.  We impose the additional 
condition that the $\theta_i$ and $\sigma_j$ lie in the half-open interval $(0,1]$.

To state the hypothesis we need a definition.  Let $(\theta_i)_{i=1}^r$, $(\sigma_j)_{j=1}^{s-1}$, 
$s\geq r$, be sequences of real numbers in the interval $(0,1]$.  Suppose first that $s=r$.  
We create two sequences of length $r$ by setting $\sigma_r=1$.  Since the definition (4.1) is 
independent of the ordering of the $\theta_i$ and $\sigma_j$, we may assume that they satisfy:
\[ \theta_1\leq\theta_2\leq \dots\leq \theta_r\quad\text{and}\quad \sigma_1\leq\sigma_2\leq\dots
\leq\sigma_r. \]
We say that $(\theta_i)_{i=1}^r$ is {\it dominated by\/} $(\sigma_j)_{j=1}^{s-1}$ if
\[ \theta_k<\sigma_k\quad\text{for $k=1,\dots,r$.} \]

Now suppose that $s>r$.  If $s=r+1$, we have two sequences of length $s-1$ and there is no need 
to augment either one.  
Assuming again that they satisfy $\theta_i\leq \theta_{i+1}$ and $\sigma_j\leq \sigma_{j+1}$ for all $i$ 
and $j$, we say that
$(\theta_i)_{i=1}^r$ is {\it dominated by\/} $(\sigma_j)_{j=1}^{s-1}$ if
\[ \theta_k<\sigma_k \quad\text{for $k=1,\dots,s-1$.} \]

If $s>r+1$, the situation is more complicated.  We combine the two sequences $(\theta_i)_{i=1}^r$ of 
length $r$  and $(i/(s-r))_{i=1}^{s-r-1}$ of length $s-r-1$ to create a sequence 
$(\tilde{\theta}_i)_{i=1}^{s-1}$ satisfying $\tilde{\theta}_i\leq \tilde{\theta}_{i+1}$ for all $i$.  It 
will be important to identify the terms of the form $k/(s-r)$ in this new sequence.  Let 
$\tilde{\theta}_{i_k} = k/(s-r)$ for $k=1,\dots,s-r-1$.  We may assume the sequence is ordered so 
that $\tilde{\theta}_{i_k}>\tilde{\theta}_{i_k-1}$ for all $k$ (unless $i_k=1$).  Supposing again that 
$\sigma_j\leq \sigma_{j+1}$ for all~$j$, we say that $(\theta_i)_{i=1}^r$ is 
{\it dominated by\/}~$(\sigma_j)_{j=1}^{s-1}$ if
\begin{equation}
\begin{split}
\tilde{\theta}_l & <  \sigma_l\quad\text{for $l\in\{1,\dots,s-1\}\setminus\{i_1,\dots i_{s-r-1}\}$,} \\
\tilde{\theta}_{i_k} & \leq\sigma_{i_k}\quad\text{for $k=1,\dots,s-r-1$.} 
\end{split}
\end{equation}

It will be convenient to have an equivalent formulation of this notion which avoids ordering the 
$\theta_i$ and $\sigma_j$.  We associate to each $j$, $j=1,\dots,s-1$, two nonnegative integers 
$I_j$, $J_j$, as follows:
\begin{align*}
I_j &= {\rm card} \{i\in\{1,\dots,r\}\mid \theta_i<\sigma_j\}, \\
J_j &= {\rm card} \{ j'\in\{1,\dots,s-1\}\mid \sigma_j\geq\sigma_{j'}\}.
\end{align*}
Note that $J_j\geq 1$ if $s>1$.  
\begin{lemma}
Let $(\theta_i)_{i=1}^r$ and $(\sigma_j)_{j=1}^{s-1}$ be sequences of $p$-integral rational numbers 
from the interval $(0,1]$ with $r\leq s$.  \\
{\bf (a)} For $r=s$, the sequence $(\theta_i)_{i=1}^r$ is {\it dominated by\/} the sequence 
$(\sigma_j)_{j=1}^{s-1}$ if and only if $\theta_i<1$ for all $i$ and $I_j\geq J_j$ for all $j$.  \\ 
{\bf (b)} For $s=r+1$, the sequence $(\theta_i)_{i=1}^r$ is {\it dominated by\/} the sequence 
$(\sigma_j)_{j=1}^{s-1}$ if and only if  $I_j\geq J_j$ for all $j$.  \\ 
{\bf (c)} For $s>r+1$, the sequence $(\theta_i)_{i=1}^r$ is {\it dominated by\/} the sequence 
$(\sigma_j)_{j=1}^{s-1}$ if and only if for all $j$, either $I_j\geq J_j$ or $\sigma_j\geq 
(J_j-I_j)/(s-r)$.
\end{lemma}

\begin{proof}
To illustrate the idea, we prove part (c).  First suppose that $(\theta_i)_{i=1}^r$ is dominated 
by $(\sigma_j)_{j=1}^{s-1}$ and that for some $j$ we have $I_j<J_j$.  We must show that 
$\sigma_j\geq(J_j-I_j)/(s-r)$.  From the definition of $I_j$ we have (assuming the $\theta_i$ 
are listed in increasing order)
\[ \theta_1\leq\dots\leq \theta_{I_j}<\sigma_j\leq \theta_{I_j+1}. \]
From (4.2) and the definition of $J_j$ we have
\[ \tilde{\theta}_1\leq \dots \leq \tilde{\theta}_{J_j}\leq\sigma_j. \]
Since $J_j>I_j$, the sequence $(\tilde{\theta}_i)_{i=1}^{J_j}$ is obtained by combining the 
sequences $(\theta_i)_{i=1}^{I_j}$ and $(k/(s-r))_{k=1}^{J_j-I_j}$.  It follows that $\sigma_j
\geq (J_j-I_j)/(s-r)$.  

Now suppose that for all $j$, either $I_j\geq J_j$ or $\sigma_j\geq(J_j-I_j)/(s-r)$.  We must 
show that $(\theta_i)_{i=1}^r$ is dominated by $(\sigma_j)_{j=1}^{s-1}$.  The sequence 
$(\tilde{\theta}_i)_{i=1}^{s-1}$ is constructed inductively.  We assume that $\theta_i\leq\theta_{i+1}$ 
for all $i$.  First of all,
\[ \tilde{\theta}_1 = \min\{\theta_1,1/(s-r)\}. \]
Assume that for some $j\geq 1$ the sequence $\tilde{\theta}_1,\dots,\tilde{\theta_j}$ has been 
constructed and that it consists of the elements $\theta_1,\dots,\theta_{j'},1/(s-r),\dots,
(j-j')/(s-r)$ listed in increasing order, where $0\leq j'\leq j$.  Then
\begin{equation}
\tilde{\theta}_{j+1} = \min\{\theta_{j'+1},(j+1-j')/(s-r)\}. 
\end{equation}

Assume that $\sigma_j\leq \sigma_{j+1}$ for all $j$.  We prove (4.2) by induction on $j$.  
Either $I_1\geq J_1$, in which case $\sigma_1>\theta_1$ since $J_1\geq 1$, or $I_1< J_1$, in which 
case $\sigma_1\geq (J_1-I_1)/(s-r)\geq 1/(s-r)$.  This proves (4.2) for $\sigma_1$.  Now suppose 
that (4.2) holds for $\sigma_j$, we must prove it for $\sigma_{j+1}$.  Since
\[ \tilde{\theta}_j = \max\{\theta_{j'},(j-j')/(s-r)\}, \]
the induction hypothesis says that
\begin{equation}
\sigma_j>\theta_{j'}\text{ and } \sigma_j\geq(j-j')/(s-r). 
\end{equation}
By (4.4), we must show that
\begin{equation}
\sigma_{j+1}>\theta_{j'+1}\text{ or } \sigma_{j+1}\geq (j+1-j')/(s-r). 
\end{equation}
If $\sigma_{j+1}>\theta_{j'+1}$, there is nothing to prove, so suppose that 
$\sigma_{j+1}\leq \theta_{j'+1}$.  The first inequality of (4.5) then implies that $I_{j+1} = j'$; 
we have trivially $J_{j+1}\geq j+1$, so $I_{j+1}<J_{j+1}$.  Our hypothesis then implies that 
\[ \sigma_{j+1}\geq (J_{j+1}-I_{j+1})/(s-r)\geq (j+1-j')/(s-r). \]
\end{proof}

Let $\alpha$ be a $p$-integral rational number, $0\leq\alpha\leq 1$.  Choose a positive integer 
$a$ such that $(p^a-1)\alpha\in{\mathbb N}$ and set $t=(p^a-1)\alpha$.  For any nonnegative 
integer $k$ we define $\alpha^{(k)} = t^{(k)}/(p^a-1)$.

\begin{proposition}
Let $(\theta_i)_{i=1}^r$ and $(\sigma_j)_{j=1}^{s-1}$ be sequences of $p$-integral rational numbers 
from the interval $(0,1]$ with $r\leq s$.  Suppose that for all $k\in{\mathbb N}$ the 
sequence $(\theta_i^{(k)})_{i=1}^r$ is dominated by the sequence $(\sigma_j^{(k)})_{j=1}^{s-1}$.  Then 
the hypergeometric series~$(4.1)$ has $p$-integral coefficients.
\end{proposition}

{\bf Remark:}  If $a$ is a positive integer such that $(p^a-1)\theta_i,(p^a-1)\sigma_j\in
{\mathbb N}$ for all~$i,j$, then $\theta_i^{(k)} = \theta_i^{(k+a)}$ and $\sigma_j^{(k)} = 
\sigma_j^{(k+a)}$ for all $i,j,k$.  It thus suffices to verify the hypothesis of Proposition~4.7 for 
$k=0,1,\dots,a-1$.

Before proving Proposition 4.7, we restate Dwork's criterion for comparison and show how Dwork's criterion follows from Proposition ~4.7 in the cases $r=s$ and $r=s-1$.  For each $i,j$, 
$1\leq i\leq r$, $1\leq j\leq s-1$, we may write (where $a$ is as in the preceding paragraph) 
\begin{align}
\theta_i &= (p^a-1)^{-1} (\theta_{i0} + \theta_{i1}p + \cdots + \theta_{i,a-1}p^{a-1}), \\
\sigma_j &= (p^a-1)^{-1} (\sigma_{j0} + \sigma_{j1}p + \cdots + \sigma_{j,a-1}p^{a-1}),
\end{align} 
where $0\leq \theta_{ik},\sigma_{jk}\leq p-1$.  For parameters in the interval $(0,1]$, one can 
use Lemma~4.3 to restate Dwork's criterion\cite[Lemma~2.2]{D2} as follows.
\begin{proposition}
Let $(\theta_i)_{i=1}^r$ and $(\sigma_j)_{j=1}^{s-1}$ be sequences of $p$-integral rational numbers 
from the interval $(0,1]$ with $r\leq s$.  Suppose that for $k=0,1,\dots,a-1$ the sequence 
$(\theta_{ik}/(p-1))_{i=1}^r$ is dominated by the sequence $(\sigma_{jk}/(p-1))_{j=1}^{s-1}$.  Then 
the hypergeometric series $(4.1)$ has $p$-integral coefficients.
\end{proposition}

{\bf Remark:}  If $s=r$ or $s=r+1$, Proposition 4.10 follows easily from Proposition 4.7 as follows:  We have from 
(4.8) and (4.9)
\[ \theta_i^{(k)} = (p^a-1)^{-1}(\theta_{ik} + \theta_{i,k+1}p + \cdots + \theta_{i,a-1}p^{a-k-1}+
\theta_{i0}p^{a-k} + \cdots + \theta_{i,k-1} p^{a-1}), \]
\[ \sigma_j^{(k)} = (p^a-1)^{-1}(\sigma_{jk} + \sigma_{j,k+1}p + \cdots + \sigma_{j,a-1}p^{a-k-1}+
\sigma_{j0}p^{a-k} + \cdots + \sigma_{j,k-1} p^{a-1}), \]
for $k=1,\dots,a-1$.  This shows that $\theta_i^{(k)}<\sigma_j^{(k)}$ if $\theta_{i,k-1}<\sigma_{j,k-1}$, 
hence the hypothesis of Proposition 4.10 implies the hypothesis of Proposition 4.7.  The examples 
of $p$-integral hypergeometric series given in \cite[Section 5]{D2} may thus be proved to be 
$p$-integral using Proposition~4.7.  When $s>r+1$, neither proposition is stronger than the other.  
For example, if $p=7$, $r=0$, $s=3$, $\sigma_1 = 3/8$, and $\sigma_2 = 5/8$, then the hypothesis 
of Proposition~4.10 is satisfied but the hypothesis of Proposition~4.7 is not.  And if $p=11$, 
$r=0$, $s=3$, $\sigma_1 = 41/120$, and $\sigma_2=81/120$, then the hypothesis of Proposition 4.7 
is satisfied but the hypothesis of Proposition 4.10 is not.

{\bf Example 2:}  Consider the series ${}_2F_1\bigg(\begin{matrix} 5/13 & 6/13 \\ & 1/2 
\end{matrix}\;\bigg|\;t\bigg)$ and let $p=3$.  We may take $a=3$ and we have 
\[ \theta_1 = 5/13,\;\theta_2=6/13;\quad \theta_1' = 6/13,\;\theta_2' = 2/13;\quad \theta_1'' = 
2/13,\;\theta_2'' = 5/13. \]
Since $\sigma_1=\sigma_1'=\sigma_1'' = 1/2$, Proposition 4.7 implies that this series has 
$3$-integral coefficients.  But we have $\theta_{12} = \theta_{22} = \sigma_{12} = 1$, so the 
hypothesis of Proposition~4.10 is not satisfied.

{\it Proof of Proposition $4.7$.}  
Let $A=\{{\bf a}_1,\dots,{\bf a}_{r+s}\}\subseteq{\mathbb R}^{r+s-1}$, where ${\bf a}_1$,\dots,
${\bf a}_{r+s-1}$ are the standard unit basis vectors and ${\bf a}_{r+s} = (1,\dots,1,-1,\dots,-1)$ 
(1 repeated $r$ times followed by $-1$ repeated $s-1$ times).  Let 
\[ v=(-\theta_1,\dots,-\theta_r,\sigma_1-1,\dots,\sigma_{s-1}-1,0)\in({\mathbb Q}\cap{\mathbb Z}_p
\cap[-1,0])^{r+s}. \]
Then $v\in R_\beta$ for $\beta = (-\theta_1,\dots,-\theta_r,\sigma_1-1,\dots,\sigma_{s-1}-1)\in
{\mathbb R}^{r+s-1}$.  It is clear that
\begin{align*}
L &= \{(-l,\dots,-l,l,\dots,l)\in{\mathbb Z}^{r+s}\mid l\in{\mathbb Z}\}, \\
L_v &=  \{(-l,\dots,-l,l,\dots,l)\in{\mathbb Z}^{r+s}\mid l\in{\mathbb N}\}.
\end{align*}
Our hypothesis implies that $0<\theta_i<1$ for $i=1,\dots,r$, so ${\rm nsupp}(v) = \emptyset$.  
In particular, $v$ has minimal negative support.  A straightforward calculation shows that
\begin{multline} 
\Phi_v(\lambda) = \\
\bigg(\prod_{i=1}^r \lambda_i^{-\theta_i}\prod_{j=1}^{s-1} \lambda_{r+j}^{\sigma_j-1}\bigg){}_rF_{s-1}\bigg(
\begin{matrix} \theta_1 & \dots & \theta_r \\ \sigma_1 & \dots & \sigma_{s-1} \end{matrix}\;\bigg|\; 
(-1)^r\pi^{s-r}\frac{\lambda_{r+1}\cdots\lambda_{r+s}}{\lambda_1\cdots\lambda_r}\bigg). 
\end{multline}

The key to proving Proposition 4.7 is the following result.
\begin{lemma}
Let $(\theta_i)_{i=1}^r$ and $(\sigma_j)_{j=1}^{s-1}$ be sequences of $p$-integral rational numbers 
from the interval $(0,1]$ with $r\leq s$.  If  $(\theta_i)_{i=1}^r$ is dominated by 
$(\sigma_j)_{j=1}^{s-1}$, then
\[ w_\Delta\big((-\beta+{\mathbb Z}^{r+s-1})\cap C(\Delta)\big) = \sum_{i=1}^r \theta_i + 
\sum_{j=1}^{s-1}(1-\sigma_j). \]
\end{lemma}

Before proving Lemma 4.12 we explain how it implies Proposition 4.7.  Lemma~4.12 and the 
definition of $v$ give
\[ w_\Delta\big((-\beta+{\mathbb Z}^{r+s-1})\cap C(\Delta)\big) = -\sum_{i=1}^{r+s} v_i. \]
The hypothesis of Proposition~4.7 thus implies that
\[ w_\Delta\big((-\beta^{(k)}+{\mathbb Z}^{r+s-1})\cap C(\Delta)\big) = -\sum_{i=1}^{r+s} v_i^{(k)} \]
for $k=0,1,\dots,a-1$.  Summing over $k$ and using (3.2) then gives
\[ w(v) = \frac{p-1}{a}\sum_{k=0}^{a-1} w_\Delta\big((-\beta^{(k)}+{\mathbb Z}^{r+s-1})\cap C(\Delta)
\big). \]
One has the trivial inequalities 
\[ w_\Delta\big((-\beta^{(k)}+{\mathbb Z}^{r+s-1})\cap \sigma_{-\beta}^\circ\big)\geq w_\Delta\big(
(-\beta^{(k)}+{\mathbb Z}^{r+s-1})\cap C(\Delta)\big) \]
and $w(v)\geq w(R_\beta)$, so Theorem 3.5 implies
\[ w(v) =  \frac{p-1}{a} \sum_{k=0}^{a-1} w_\Delta\big((-\beta^{(k)}+{\mathbb Z}^{r+s-1})\cap 
\sigma_{-\beta}^\circ\big). \]
By Corollary 3.6, the series $\Phi_v(\lambda)$ has $p$-integral coefficients.  Proposition 4.7 
now follows from (4.11).

\begin{proof}[Proof of Lemma $4.12$]
Let $u_1,\dots,u_{r+s-1}$ be coordinates on ${\mathbb R}^{r+s-1}$.  The polytope~$\Delta$, the convex 
hull of $A\cup\{{\bf 0}\}$, has a vertex at the origin.  If $r=s$, it has a unique codimension-one 
face not containing the origin, which we denote by $\Delta_0$.  This face has vertices ${\bf a}_1,
\dots,{\bf a}_{r+s}$ and lies in the hyperplane $\sum_{k=1}^{r+s-1} u_k=1$.  If $s>r$ there are $s$ 
codimension-one faces not containing the origin, all of them simplices, which we denote by 
$\Delta_0,\Delta_1,\dots,\Delta_{s-1}$.  The face $\Delta_0$ has vertices ${\bf a}_1,\dots,
{\bf a}_{r+s-1}$ and lies in the hyperplane $\sum_{k=1}^{r+s-1} u_k=1$.  For $i=1,\dots,s-1$, the set 
of vertices of $\Delta_i$ is $A\setminus\{{\bf a}_{r+i}\}$ and it lies in the hyperplane 
\[ \sum_{\substack{k=1\\ k\neq r+i}}^{r+s-1} u_k - (s-r-1) u_{r+i} = 1. \]

Since $-\beta\in C(\Delta_0)$, we have 
\[ w_\Delta(-\beta) = w_{\Delta_0}(-\beta) = \sum_{i=1}^r \theta_i + \sum_{j=1}^{s-1} (1-\sigma_j). \]
Thus the assertion of Lemma 4.12 is that the weight function $w_\Delta$ has  its minimum on 
$(-\beta+{\mathbb Z}^{r+s-1})\cap C(\Delta)$ at the point $-\beta$.  Let $\gamma=(\gamma_1,\dots,
\gamma_{r+s-1})\in(-\beta+{\mathbb Z}^{r+s-1})\cap C(\Delta)$.  We must show that
\begin{equation}
w_\Delta(\gamma)\geq \sum_{i=1}^r \theta_i + \sum_{j=1}^{s-1} (1-\sigma_j).
\end{equation}

Write $\gamma = -\beta + z$ with $z=(z_1,\dots,z_{r+s-1})\in{\mathbb Z}^{r+s-1}$, so that
\begin{equation}
\gamma_i = \theta_i + z_i\quad\text{for $i=1,\dots,r$}
\end{equation}
and
\begin{equation}
\gamma_{r+j} = 1-\sigma_j + z_{r+j}\quad\text{for $j=1,\dots,s-1$.}
\end{equation}
We divide the proof into two cases according as to whether $r=s$ (the nonconfluent case) or 
$r<s$ (the confluent case).
  
Consider first the case $r=s$, where $\Delta_0$ is the unique codimension-one face of $\Delta$ 
not containing the origin.  Then
\begin{equation}
w_\Delta(\gamma) = \sum_{i=1}^{r+s-1}\gamma_i = \sum_{i=1}^r\theta_i +\sum_{j=1}^{s-1}(1-\sigma_j) + 
\sum_{k=1}^{r+s-1} z_k.
\end{equation}
Thus to show (4.13), we need to show that
\begin{equation}
\sum_{k=1}^{r+s-1} z_k\geq 0
\end{equation}

The elements of the set $A$ satisfy the inequalities
\begin{equation} 
u_i\geq 0 \quad\text{for $i=1,\dots,r$} 
\end{equation}
and
\begin{equation}
u_i+u_j\geq 0 \quad\text{for $i=1,\dots,r$ and $j=r+1,\dots,r+s-1$}, 
\end{equation}
so all elements of $C(\Delta)$ satisfy these inequalities as well.  
In particular, $\gamma$ must satisfy these inequalities.  By hypothesis there are orderings 
of $\{1,\dots,r-1\}$ such that
\begin{equation}
\theta_{i_1}<\sigma_{j_1},\dots,\;\theta_{i_{r-1}}<\sigma_{j_{r-1}},\;\theta_{i_r}<1. 
\end{equation}
It follows from (4.14), (4.15), (4.18), and (4.19) that 
\begin{equation}
\theta_{i_r} + z_{i_r}\geq 0
\end{equation}
and that
\begin{equation}
\theta_{i_k} + z_{i_k} + (1-\sigma_{j_k}) + z_{r+{j_k}}\geq 0\quad\text{for $k=1,\dots,s-1$.}
\end{equation}
Since the $z_i,z_j$ are integers, Equations (4.20), (4.21), and (4.22) imply that
\[ z_{i_r}\geq 0 \]
and
\[ z_{i_k} + z_{r+j_k}\geq 0\quad\text{for $k=1,\dots,s-1$.} \]
Summing over all indices then gives
\[ \sum_{i=1}^{r+s-1} z_i = z_{i_r}+\sum_{k=1}^{s-1} (z_{i_k} + z_{r+j_k})\geq 0, \]
which establishes (4.17) and completes the proof of Lemma 4.12 when $r=s$.

Now suppose that $s>r$.  For $i=0,1,\dots,s-1$, let $C(\Delta_i)$ be the real cone generated by 
$\Delta_i$.  For $\gamma\in C(\Delta)$, one has $\gamma\in C(\Delta_i)$ for some $i$ and 
$w_\Delta(\gamma) = w_{\Delta_i}(\gamma)$.  Suppose first that $\gamma\in C(\Delta_0)$.  Then
\[ w_{\Delta_0}(\gamma) = \sum_{i=1}^r \theta_i + \sum_{j=1}^{s-1}(1-\sigma_j) + \sum_{k=1}^{r+s-1} z_k. \]
Since $C(\Delta_0)$ is the first orthant in ${\mathbb R}^{r+s-1}$, the condition $\gamma\in 
C(\Delta_0)$ and Equations~(4.14) and~(4.15) imply that $z_k\geq 0$ for all $k$ so (4.13) holds in 
this case.  

Now suppose that $\gamma\in C(\Delta_i)$ with $1\leq i\leq s-1$.  To fix ideas, we take $i=1$ and 
prove~(4.13) for $\gamma\in C(\Delta_1)$.  Since $C(\Delta_1)$ is the cone generated by 
$A\setminus\{{\bf a}_{r+1}\}$, we can write $\gamma$ (uniquely) as a linear combination with 
nonnegative coefficients of the vectors in that set.  Explicitly,
\begin{multline}
\gamma =(\sigma_1-1-z_{r+1}){\bf a}_{r+s} +  \sum_{i=1}^r \big((\theta_i+z_i) - (\sigma_1-1-z_{r+1})
\big){\bf a}_i \\ +\sum_{j=2}^{s-1} \big( (\sigma_1-1-z_{r+1}) - (\sigma_j - 1 -z_{r+j})\big){\bf a}_{r+j}. 
\end{multline}
The condition that $\gamma\in C(\Delta_1)$ is equivalent to the condition that the coefficients 
of the ${\bf a}_i$, $i\in A\setminus\{{\bf a}_{r+1}\}$, in~(4.23) are all nonnegative.  Furthermore, 
$w_{\Delta_1}(\gamma)$ is the sum of these coefficients:
\begin{equation}
w_{\Delta_1}(\gamma) = \sum_{i=1}^r (\theta_i+z_i) + \sum_{j=2}^{s-1}(1-\sigma_j+z_{r+j}) + (r-s+1)
(1-\sigma_1 + z_{r+1}).
\end{equation}

Consider the coefficients in (4.23).  If $\sigma_1-1-z_{r+1}=0$, then $\gamma\in C(\Delta_0)$ and we 
have already verified (4.13) in that case.  So we assume $\sigma_1-1-z_{r+1}>0$, which implies
\begin{equation}
z_{r+1}\leq -1.
\end{equation}
For $i=1,\dots,r$, we have $(\theta_i+z_i) - (\sigma_1-1-z_{r+1})\geq 0$, which implies
\begin{equation}
z_i\geq\begin{cases} -z_{r+1} & \text{if $\theta_i<\sigma_1$,} \\
-z_{r+1}-1 & \text{if $\theta_i\geq\sigma_1$.} \end{cases}
\end{equation}
For $j=2,\dots,s-1$, we have $(\sigma_1-1-z_{r+1}) - (\sigma_j - 1 -z_{r+j})\geq 0$, which implies
\begin{equation}
z_{r+j}\geq\begin{cases} z_{r+1} & \text{if $\sigma_1\geq\sigma_j$,} \\
z_{r+1}+1 & \text{if $\sigma_1<\sigma_j$.} \end{cases}
\end{equation}
Using (4.26) and (4.27) in (4.24) and simplifying gives
\begin{equation}
w_{\Delta_1}(\gamma)\geq w_\Delta(-\beta) + (s-r)\sigma_1-z_{r+1}+I_1-J_1-1.
\end{equation}
Equation (4.25) implies
\begin{equation}
w_{\Delta_1}(\gamma)\geq w_\Delta(-\beta) + (s-r)\sigma_1 +I_1-J_1.
\end{equation}
If $I_1\geq J_1$, then $w_{\Delta_1}(\gamma)> w_\Delta(-\beta)$ since $(s-r)\sigma_1>0$.  If $I_1<J_1$, 
then $w_{\Delta_1}(\gamma)\geq w_\Delta(-\beta)$ by Lemma~4.3(c).  This completes the proof of (4.13) 
for $\gamma\in C(\Delta_1)$.  
\end{proof}

{\bf Remark:}  If the $\theta_i, \sigma_j$ satisfy the analogous strict inequalities rather than (4.2) itself, then the above argument shows (in the confluent case $s>r$) that the point $-\beta$ is the unique element of 
$(-\beta+{\mathbb Z}^{r+s-1})\cap C(\Delta)$ satisfying
\[ w_{\Delta}(-\beta) = w_\Delta\big((-\beta+{\mathbb Z}^{r+s-1})\cap C(\Delta)\big), \]
i.~e., inequality (4.13) is strict unless $\gamma = -\beta$.  We believe that this uniqueness 
property of $-\beta$ is sufficient to imply the $p$-adic analytic continuation result for ratios of 
hypergeometric series that follows from Dwork\cite[Theorems 1.1 and 3.1]{D2}.

\section{$A$-hypergeometric series with integer coefficients}

In this section we suppose we are in the nonconfluent situation: there exists a linear form $h$ on 
${\mathbb R}^n$ such that $h({\bf a}_i) = 1$ for $i=1,\dots,N$.  In this case the normalizing factor 
$\pi$ cancels and the series $\Phi_v(\lambda)$ is independent of $p$:
\begin{equation}
\Phi_v(\lambda) = \sum_{l\in L_v}\frac{[v]_{l_-}}{[v+l]_{l_+}}\lambda^{v+l}. 
\end{equation}
Let $v=(v_1,\dots,v_N)\in R_\beta$ and suppose that for all $i$ either $v_i=0$ or $v_i=-1$.  The 
coefficients of $\Phi_v(\lambda)$ lie in ${\mathbb Q}$ and the $v_i$ are $p$-integral for all 
primes $p$.  We study conditions that guarantee the coefficients of $\Phi_v(\lambda)$ lie in 
${\mathbb Z}$.

To fix ideas, suppose that for some $M$, $0\leq M\leq N$, we have
\[ v_i = \begin{cases} -1 & \text{for $i=1,\dots,M$,} \\ 0 & \text{for $i=M+1,\dots,N$.} \end{cases} 
\]
Then $w(v) = M(p-1)$ and $\beta=-\sum_{i=1}^M {\bf a}_i\in{\mathbb Z}^n$.  
From Corollary~3.6 we get immediately the following result.
\begin{proposition}
If $w_\Delta(\sigma_{-\beta}^\circ\cap{\mathbb Z}^n) = M$, then $\Phi_v(\lambda)$ has integer coefficients.
\end{proposition}

{\bf Remark:}  Note that $\Delta$ is a pyramid with vertex at the origin and base in the hyperplane 
$h=1$.  It follows that for $\gamma\in C(\Delta)$ one has $w_\Delta(\gamma) = h(\gamma)$.  In 
particular, $w_\Delta(-\beta) = h(-\beta) = M$.  Thus Proposition~5.2 says that $\Phi_v(\lambda)$ 
will have integer coefficients if the linear form $h$ assumes its minimum value on 
$\sigma_{-\beta}^\circ\cap{\mathbb Z}^n$ at~$-\beta$.

{\bf Example 3:}  Consider the set $A= \{ {\bf a}_1,\dots,{\bf a}_6\}\subseteq{\mathbb R}^4$ given by
\[ {\bf a}_1 = (2,0,0,1),\;{\bf a}_2 = (0,2,0,1),\;{\bf a}_3 = (0,0,2,1), \]
\[ {\bf a}_4 = (1,1,0,1),\;{\bf a}_5 = (1,0,1,1),\;{\bf a}_6 = (0,1,1,1). \]
One computes that
\[ L = \{(l_1,l_2,l_3,-l_1-l_2+l_3,-l_1+l_2-l_3,l_1-l_2-l_3)\mid l_1,l_2,l_3\in{\mathbb Z}\}. \]
Take $v=(0,0,0,-1,-1,0)$, giving $\beta = (-2,-1,-1,-2)$.  One computes that
\begin{multline*} L_v = \\ 
\{(l_1,l_2,l_3,-l_1-l_2+l_3,-l_1+l_2-l_3,l_1-l_2-l_3)\mid \text{$l_1,l_2,l_3\in{\mathbb N}$ and $l_1\geq l_2+l_3$}\}. 
\end{multline*}
This gives
\begin{multline*}
\Phi_v(\lambda) = (\lambda_4\lambda_5)^{-1} \\ 
\cdot\sum_{l_2,l_3=0}^\infty \sum_{l_1=l_2+l_3}^\infty \frac{(l_1+l_2-l_3)!(l_1-l_2+l_3)!}{l_1!l_2!l_3!(l_1-l_2-l_3)!} \bigg(\frac{\lambda_1\lambda_6}{\lambda_4\lambda_5}\bigg)^{l_1} \bigg(\frac{\lambda_2\lambda_5}{\lambda_4\lambda_6}\bigg)^{l_2} \bigg(\frac{\lambda_3\lambda_4}{\lambda_5\lambda_6}\bigg)^{l_3}. 
\end{multline*}
The set $\Delta$ is the convex hull of the origin and the points ${\bf a}_1$, ${\bf a}_2$, ${\bf a}_3$, so $-\beta$ is an interior point of $C(\Delta)$ and $\sigma_{-\beta}^\circ$ is the interior of $C(\Delta)$.  If we let $u_1,\dots,u_4$ be the coordinates on ${\mathbb R}^4$, then the linear form $h(u) = u_4$ satisfies $h({\bf a}_i)=1$ for $i=1,\dots,6$.  There are no interior lattice points $u\in C(\Delta)$ with $h(u)=1$, so $h$ assumes its minimum value of $2$ on $\sigma_{-\beta}^\circ\cap{\mathbb Z}^4$ at the point $-\beta$.  It now follows from Proposition~5.2 that the series $\Phi_v(\lambda)$ has integer coefficients.  Note that this does not follow trivially from the form of the coefficients, i.~e., they are not binomial coefficients.

We examine the $A$-hypergeometric systems associated to complete intersections in the $n$-torus 
${\mathbb T}^n$.  This requires changing our notation somewhat from the rest of the paper for the 
remainder of this section.  Consider sets $A_i = \{{\bf a}_1^{(i)},\dots,{\bf a}_{N_i}^{(i)}\}\subseteq
{\mathbb Z}^n$, $i=1,\dots,M$.  We consider the family of complete intersections in the $n$-torus 
${\mathbb T}^n$ over ${\mathbb C}$ defined by the equations
\[ f_{i,\lambda}(x):=\sum_{j=1}^{N_i} \lambda_{ij}x^{{\bf a}_j^{(i)}} = 0\quad\text{for $i=1,\dots,M$,} \]
where the $\lambda_{ij}$ are indeterminates.  

We define $\hat{\bf a}_j^{(i)}\in{\mathbb Z}^{n+M}$ for $i=1,\dots,M$ by
\[ \hat{\bf a}_j^{(i)} = ({\bf a}_j^{(i)},0,\dots,0,1,0,\dots,0), \]
where the ``1'' occurs in the $(n+i)$-th entry, and we consider the $A$-hyper\-geometric system 
associated to the set
\[ A:= \{\hat{\bf a}_j^{(i)}\mid i=1,\dots,M,\; j=1,\dots,N_i\}. \]
Note that if we let $u_1,\dots,u_{n+M}$ be the coordinate functions on ${\mathbb R}^{n+M}$ and set $h(u) = \sum_{i=1}^M u_{n+i}$, then all $\hat{\bf a}_j^{(i)}$ lie in the hyperplane $h(u) = 1$, so we are in the nonconfluent situation.

Fix $M'\leq M$.  For each $i=1,\dots,M'$, we distinguish a monomial by choosing ${\bf a}_{j_i}^{(i)}\in A_i$.  We shall 
associate to the sequence $({\bf a}_{j_i}^{(i)})_{i=1}^{M'}$ of distinguished monomials an $A$-hypergeometric series with integer 
coefficients.   Set $-\beta = \sum_{i=1}^{M'} {\bf a}_{j_i}^{(i)}$,  set $v=(v_j^{(i)})$, $i=1,\dots,M$, 
$j=1,\dots,N_i$, where
\begin{equation}
 v^{(i)}_j = \begin{cases} -1 & \text{if $i\in\{1,\dots,M'\}$ and $j=j_i$,} \\ 0 & \text{otherwise,} \end{cases} 
\end{equation}
and define
\[ \hat{\beta} = \sum_{i=1}^M\sum_{j=1}^{N_i} v^{(i)}_j\hat{\bf a}_j^{(i)} = (\beta,-1,\dots,-1,0,\dots,0)\in
{\mathbb Z}^{n+M} \]
(``$-1$'' appears $M'$ times, ``$0$'' appears $M-M'$ times).
\begin{proposition}
The vector $v\in R_{\hat{\beta}}$ has minimal negative support and the series $\Phi_v(\lambda)$ is a 
solution with integer coefficients of the $A$-hypergeometric system with parameter $\hat{\beta}$.
\end{proposition}

\begin{proof}
Suppose that $l=(l_j^{(i)})_{i,j}\in L$ has $l_{j_I}^{(I)}>0$ for some $I\in\{1,\dots,M'\}$.  We have 
$\sum_{i,j} l_j^{(i)}\hat{\bf a}_j^{(i)} = {\bf 0}$; the vanishing of the $(n+I)$-th coordinate on the 
left-hand side implies that $l_{j'}^I<0$ for some $j'\neq j_I$.  It follows that ${\rm nsupp}(v+l)$ 
cannot be a proper subset of ${\rm nsupp}(v)$.

Fix a prime $p$.  We clearly have $w(v) = M'(p-1)$.  Note that 
\[ \sigma_{-\hat{\beta}}^\circ\subseteq C(\Delta)\cap\{ (u_1,\dots,u_{n+M})\in{\mathbb R}^{n+M}\mid \text{$u_{n+i}>0$ for $i=1,\dots,M'$}\}. \]
It follows that $h$ cannot assume any value less than $M'$ at a lattice point in $\sigma^\circ_{-\hat{\beta}}$, so $h$ assumes its minimum value of $M'$ on $\sigma^\circ_{-\hat{\beta}}\cap{\mathbb Z}^{n+M}$ at the point $-\hat{\beta}$.  Proposition~5.2 now implies that $\Phi_v(\lambda)$ has integer coefficients.
\end{proof}

{\bf Example 4:}  Consider the complete intersection defined by the equations
\begin{align*}
f_{1,\lambda}(x_1,\dots,x_6) &= \lambda_{11}x_1^3+\lambda_{12}x_2^3 +\lambda_{13}x_3^3 + 
\lambda_{14} x_4x_5x_6 = 0 \\
f_{2,\lambda}(x_1,\dots,x_6) &= \lambda_{21}x_4^3+\lambda_{22}x_5^3 +\lambda_{22}x_6^3 + 
\lambda_{24} x_1x_2x_3 = 0.
\end{align*} 
For notational convenience we let ${\bf e}_1,\dots,{\bf e}_6$ denote the standard unit basis 
vectors in ${\mathbb R}^6$.  Then
$A_1 = \{{\bf a}_1^{(1)},\dots,{\bf a}_4^{(1)}\}$, where
\[ {\bf a}_1^{(1)} = 3{\bf e}_1,\;{\bf a}_2^{(1)} = 3{\bf e}_2,\;{\bf a}_3^{(1)} = 3{\bf e}_3,\;
{\bf a}_4^{(1)} = {\bf e}_4+{\bf e}_5+{\bf e}_6, \]
and $A_2 = \{{\bf a}_1^{(2)},\dots,{\bf a}_4^{(2)}\}$, where
\[ {\bf a}_1^{(2)} = 3{\bf e}_4,\;{\bf a}_2^{(2)} = 3{\bf e}_5,\;{\bf a}_3^{(2)} = 3{\bf e}_6,\;
{\bf a}_4^{(2)} = {\bf e}_1+{\bf e}_2+{\bf e}_3. \]
We consider the series associated to the sequence $({\bf a}^{(1)}_4,{\bf a}^{(2)}_4)$.  One has
$v = (v_j^{(i)})$, $i=1,2$, $j=1,\dots,4$, with $v_4^{(1)} = -1$, $v_4^{(2)} = -1$, $v_j^{(i)}=0$ for all 
other $i,j$, and 
\[ \hat{\beta} = \sum_{i=1}^2\sum_{j=1}^4 v_j^{(i)}\hat{\bf a}_j^{(i)} = (-1,\dots,-1)\in{\mathbb R}^8. \]
Proposition~5.4 implies that $\Phi_v(\lambda)$ has integer coefficients.  Note that the lattice of 
relations $L$ has the following form: $l=(l_j^{(i)})\in L$ if and only if $l$ is an integer multiple 
of the element having $l_j^{(i)} = 1$ for $j=1,2,3$ and $i=1,2$, $l_4^{(1)} = l_4^{(2)} = -3$.  The set 
$L_v$ consists of all nonnegative integer multiples of this element.  It follows that 
\begin{equation}
\Phi_v(\lambda) = \big(\lambda_4^{(1)}\lambda_4^{(2)}\big)^{-1}\sum_{l=0}^\infty \frac{(3l)!^2}{l!^6}
\bigg(\frac{\lambda_1^{(1)}\lambda_2^{(1)}\lambda_3^{(1)} \lambda_1^{(2)}\lambda_2^{(2)}\lambda_3^{(2)}}
{\big(\lambda_4^{(1)}\lambda_4^{(2)}\big)^3}\bigg)^l.
\end{equation}

If we write the system in one-parameter form as
\begin{align*}
x_1^3+x_2^3 +x_3^3 -3\lambda x_4x_5x_6 &= 0 \\
x_4^3+x_5^3 +x_6^3 -3\lambda x_1x_2x_3 &= 0,
\end{align*} 
then our solution specializes to 
\[ \Phi_v(1,1,1,-3\lambda,1,1,1,-3\lambda) = (3\lambda)^{-2} \sum_{l=0}^{\infty} \frac{(3l)!^2}{l!^6}
(3\lambda)^{-6l}. \]
(see Libgober-Teitelbaum\cite{LT}).

For all other sequences $({\bf a}^{(1)}_{j_1},{\bf a}^{(2)}_{j_2})$ in this example one has $L_v = 
\{{\bf 0}\}$, so the corresponding series $\Phi_v(\lambda)$ consists of just a single monomial.

\section{Parameters mod ${\mathbb Z}^n$}

In this section we consider the system (1.1), (1.2) when the parameter varies in a shifted lattice 
$\beta+{\mathbb Z}^n$.  Put
\[ RA = \bigg\{ \sum_{i=1}^N r_i{\bf a}_i\;\bigg|\; r=(r_1,\dots,r_N)\in R\bigg\}. \]
Since $\beta+{\mathbb Z}^n$ is discrete and $RA$ is bounded, their intersection is a finite set. Put 
\[ \Gamma = (-\sigma_{-\beta}^\circ)\cap(\beta+{\mathbb Z}^n)\cap RA. \]  
We assume $\beta$ to be chosen so that this intersection is nonempty.   Let $B={\rm card}(\Gamma)$.
Fix a positive integer $a$ such that $(p^a-1)\beta\in{\mathbb Z}^n$.  Then $(p^a-1)\gamma\in
{\mathbb Z}^n$ for all $\gamma\in\Gamma$.  Recall that in Section 2 we defined for $b\in{\mathbb N}_+$
\[ R_{\gamma,b} = \{r\in R_\gamma\mid (1-p^b)r\in{\mathbb N}^N\}.  \]
\begin{theorem}
There exist $\gamma\in\Gamma$, a positive integer $b\leq B$, and $v\in R_{\gamma,ab}$ such that
\[ w(v) = w(R_{\gamma,ab}) = w(R_{\gamma})= \min\{ w(R_\delta)\mid \delta\in\Gamma \}. \]
In particular, the series $\Phi_v(\lambda)$ has $p$-integral coefficients.
\end{theorem}

Since $R_{\gamma,ab}$ is a finite set, the existence of $v\in R_{\gamma,ab}$ satisfying the first 
equality of Theorem~6.1 is trivial.  The proof of the latter two equalities of Theorem~6.1 is 
based on the following lemma.
\begin{lemma}
For all $\gamma\in\Gamma$ and every positive integer $b$
\[ w(R_{\gamma,ab})\geq \min\{w(R_{\delta,ac})\mid \delta\in\Gamma\text{ and }1\leq c\leq B\} . \]
\end{lemma}

We first observe that Lemma 6.2 implies Theorem 6.1 and then prove Lemma~6.2.
Since $R_{\gamma} = \bigcup_{b=1}^\infty R_{\gamma,ab}$, Lemma 6.2 implies that 
\[ w(R_{\gamma})\geq \min\{w(R_{\delta,ac})\mid \delta\in\Gamma\text{ and }1\leq c\leq B\} \]
for all $\gamma\in\Gamma$.  On the other hand, $w(R_{\delta,ac})\geq w(R_{\delta})$ for all 
$\delta\in\Gamma$ and all $c$, hence
\[ w(R_\gamma)\geq \min\{w(R_{\delta,ac})\mid \delta\in\Gamma\text{ and }1\leq c\leq B\} \geq 
\min\{w(R_\delta)\mid \delta\in\Gamma\} \]
for all $\gamma\in\Gamma$.
We therefore have
\[ \min\{w(R_{\delta,ac})\mid \delta\in\Gamma\text{ and }1\leq c\leq B\} = \min\{w(R_\delta)\mid 
\delta\in\Gamma\}. \]
It follows that if we choose $\gamma\in\Gamma$ and $b\leq B$ such that
\[ w(R_{\gamma,ab}) = \min\{w(R_{\delta,ac})\mid \delta\in\Gamma\text{ and }1\leq c\leq B\}, \]
then
\begin{equation}
w(R_{\gamma,ab}) = w(R_{\gamma}) = \min\{w(R_{\delta})\mid \delta\in\Gamma\},
\end{equation}
which proves Theorem 6.1.

\begin{proof}[Proof of Lemma $6.2$]
Fix $\gamma\in\Gamma$ and a positive integer $b$.  It suffices by induction to show that if $b>B$, 
then there exist $\delta\in\Gamma$ and $0<t<b$ such that
\begin{equation}
w(R_{\gamma,ab})\geq \min\{w(R_{\delta,at}),w(R_{\delta,a(b-t)})\}.
\end{equation}
To prove (6.4) it suffices to show that, given $r\in R_{\gamma,ab}$, there exist $\delta\in\Gamma$, 
$u\in R_{\delta,at}$, and
$v\in R_{\delta,a(b-t)}$ such that 
\begin{equation}
w(r)\geq\min\{w(u),w(v)\}.
\end{equation}  
Recall that in Section 3 we defined maps $r\mapsto r^{(k)}$ of $R$ to $R$. 
Let $\delta = \sum_{i=1}^N r^{(a)}_i{\bf a}_i$, so that $r^{(a)}\in R_{\delta,ab}$.  By definition $r^{(a)}
\in R$, so $\delta\in RA$.  Since $(p^a-1)\beta\in{\mathbb Z}^n$, we have $\delta\in\beta+
{\mathbb Z}^n$.  Finally, by Lemma~3.3, we have $\delta\in -\sigma_{-\beta}^\circ$, so $\delta\in\Gamma$.

Consider $r^{(ak)}$ for $k=0,1,\dots,b-1$.  By what we have just proved, we have
$r^{(ak)}\in R_{\delta_k,ab}$ for some $\delta_k\in\Gamma$.  The set $\Gamma$ has cardinality~$B$, so 
$b>B$ implies there exist $0\leq k_1<k_2\leq b-1$ such that 
\begin{equation}
\sum_{i=1}^N r^{(ak_1)}_i{\bf a}_i = \sum_{i=1}^N r^{(ak_2)}_i{\bf a}_i\in \Gamma.
\end{equation}
To simplify notation we set $s=r^{(ak_1)}$, $t=k_2-k_1$, and $\epsilon = \sum_{i=1}^N s_i{\bf a}_i$, so 
that (6.6) implies
\begin{equation}
\sum_{i=1}^N s_i{\bf a}_i = \sum_{i=1}^N s^{(at)}_i{\bf a}_i=\epsilon.
\end{equation}
Note that $w(r^{(ak)})=w(r)$ for all $k$, so to prove (6.5) it suffices to show there exist 
$u\in R_{\epsilon,at}$ and $v\in R_{\epsilon,a(b-t)}$ such that
\begin{equation}
w(s)\geq\min\{w(u),w(v)\}.
\end{equation}

Set $q=p^a$ and write
\[ (1-q^b)s_i = s_{i0} + s_{i1}q+\cdots+s_{i,b-1}q^{b-1} \]
with $0\leq s_{ik}\leq q-1$ for $k=0,\dots,b-1$.  Then $s_i^{(ak)}$ satisfies
\[ (1-q^b)s^{(ak)}_i = s_{ik} + s_{i,k+1}q+\cdots+s_{i,b-1}q^{b-1-k} + s_{i0}q^{b-k} +\cdots+ s_{i,k-1}q^{b-1} \]
for $k=1,\dots,b$.  Define
\[ \epsilon^{(ak)} = \sum_{i=1}^N s_i^{(ak)}{\bf a}_i. \]
A calculation shows that
\begin{equation}
\epsilon^{(ak)}-q\epsilon^{(a(k+1))} = \sum_{i=1}^N s_{ik}{\bf a}_i.
\end{equation}
It follows from (6.7) that $\epsilon^{(at)} = \epsilon$, hence by (6.9)
\begin{align*}
\sum_{i=1}^N (s_{i0} + s_{i1}q+\cdots+s_{i,t-1}q^{t-1}){\bf a}_i &= \sum_{k=0}^{t-1} q^k(\epsilon^{(ak)} - 
q\epsilon^{(a(k+1))}) \\
 &= \epsilon-q^t\epsilon^{(at)} = (1-q^t)\epsilon.
\end{align*}
Similarly we have
\begin{align*}
\sum_{i=1}^N (s_{it} + s_{i,t+1}q+\cdots+s_{i,b-1}q^{b-1-t}){\bf a}_i &= \sum_{k=t}^{b-1} 
q^{k-t}(\epsilon^{(ak)} - q\epsilon^{(a(k+1))}) \\
 &= \epsilon^{(at)}-q^{b-t}\epsilon^{(ab)} = (1-q^{b-t})\epsilon.
\end{align*}

Define for $i=1,\dots,N$
\begin{align*}
u_i &= (1-q^t)^{-1}\sum_{k=0}^{t-1} s_{ik}q^{k}, \\
v_i &= (1-q^{b-t})^{-1}\sum_{k=t}^{b-1} s_{ik}q^{k-t}.
\end{align*}
The above equations give
\[ \sum_{i=1}^N u_i{\bf a}_i = \sum_{i=1}^N v_i{\bf a}_i = \epsilon. \]
It follows that $u=(u_1,\dots,u_N)\in R_{\epsilon,at}$ and $v=(v_1,\dots,v_N)\in R_{\epsilon,a(b-t)}$.  
From the definitions of $u$ and $v$ we have
\[ w_p((1-q^b)s) = w_p((1-q^{t})u) + w_p((1-q^{b-t})v), \]
hence
\[ \frac{w_p((1-q^b)s)}{b} = \bigg(\frac{t}{b}\bigg)\frac{w_p((1-q^{t})u)}{t} + \bigg(1-\frac{t}{b}
\bigg)\frac{w_p((1-q^{b-t})v)}{b-t}. \]  
This gives $w(s) = (t/b)w(u) + (1-t/b)w(v)$, which implies (6.8).
\end{proof}

\section{Exponential sums}

Let ${\mathbb F}_q$ be the finite field of $q=p^a$ elements and let $\bar{\lambda}_1,\dots,\bar{\lambda}_N\in{\mathbb F}_q$.  Define
\[ f_{\bar{\lambda}}(x) = \sum_{j=1}^N \bar{\lambda}_j x^{{\bf a}_j}\in{\mathbb F}_q[x_1^{\pm 1},\dots,x_n^{\pm 1}]. \]
Let $\Psi:{\mathbb F}_p\to{\mathbb Q}_p(\zeta_p)^\times$ be the nontrivial additive character satisfying $\Psi(1) \equiv 1+\pi \pmod{\pi^2}$ and let $\omega:{\mathbb F}_q^\times\to{\mathbb Q}_p(\zeta_{q-1})^\times$ be the Teichm\"uller multiplicative character.  Let ${\bf e} = (e_1,\dots,e_n)\in{\mathbb Z}^n$.  We consider the exponential sum 
\begin{multline}
S(f_{\bar{\lambda}},{\bf e}) = \\ 
\sum_{x=(x_1,\dots,x_n)\in({\mathbb F}_q^\times)^n} \omega(x_1)^{-e_1}\cdots \omega(x_n)^{-e_n}\Psi\big({\rm Tr}_{{\mathbb F}_q/{\mathbb F}_p}(f_{\bar{\lambda}}(x))\big)\in{\mathbb Q}_p(\zeta_p,\zeta_{q-1}).
\end{multline}
(Of course, if ${\bf e}\equiv{\bf e'}\pmod{(q-1){\mathbb Z}^n}$, then $S(f_{\bar{\lambda}},{\bf e}) = S(f_{\bar{\lambda}},{\bf e'})$.) 

We put $\beta = -{\bf e}/(q-1)$, $M={\bf e}+(q-1){\mathbb Z}^n$, and $\Gamma=(\beta + {\mathbb Z}^n)\cap RA$, a finite set.  For $b\in{\mathbb N}_+$ define $R_b\subseteq R$ by
\[ R_b = \{r\in R\mid (1-p^b)r\in{\mathbb N}^N\}. \]
Clearly $R = \bigcup_{b\in{\mathbb N}_+} R_b$.  We also define $\Gamma_b\subseteq\Gamma$ by $\Gamma_b = (\beta+{\mathbb Z}^n)\cap R_bA$.  In \cite[Section 5]{AS3}, we associated to $M$ a set $U_M\subseteq{\mathbb N}^N$, namely,
\[ U_M = \bigg\{u=(u_1,\dots,u_N)\in\{0,1,\dots,q-1\}^N \mid \sum_{j=1}^N u_j{\bf a}_j\in M\bigg\}. \]
A short calculation shows that one has $u=(u_1,\dots,u_N)\in U_M$ if and only if $(1-q)^{-1}u\in R_a$ and
\[ \sum_{j=1}^N \frac{u_j}{1-q}{\bf a}_j \in \beta + {\mathbb Z}^n. \]
Since $q=p^a$, it follows that
\begin{equation}U_M = \bigcup_{\gamma\in\Gamma_a} (1-p^a)R_{\gamma,a}\quad\text{(a disjoint union),} 
\end{equation}
where 
\[ (1-p^a)R_{\gamma,a} = \{ (1-p^a)r\mid r\in R_{\gamma,a}\}. \]

In \cite[Section 5]{AS3} we defined the {\it $p$-weight of $M$\/}, $w_p(M)$, to be
\[ w_p(M) = \min\{ w_p(u)\mid u\in U_M\}. \]
It follows from (7.2) that $w_p(M) = \min_{\gamma\in\Gamma_a} \{aw(R_{\gamma,a}\}$.  Put
\[ \Gamma_a' = \{\gamma\in\Gamma_a\mid aw(R_{\gamma,a}) = w_p(M)\}. \]
By \cite[Equation (6.3)]{AS3} we then have
\begin{multline}
S(f_{\bar{\lambda}},{\bf e})\equiv \\
(-1)^n\bigg(\sum_{\gamma\in\Gamma'_a} \sum_{\substack{r\in R_{\gamma,a}\\ w(r) = w(R_{\gamma,a})}} \frac{\prod_{j=1}^N \omega(\bar{\lambda_j})^{(1-p^a)r_j}}{\big((1-p^a)r\big)!!} \bigg)\pi^{w_p(M)} \pmod{\pi^{w_p(M)+1}}.
\end{multline}

For $\gamma\in\Gamma'_a$, let $v(\gamma)\in R_{\gamma,a}$ be an element such that $w(v(\gamma)) = w(R_{\gamma,a})$.  Then by (2.17)
\begin{multline}
S(f_{\bar{\lambda}},{\bf e})\equiv (-1)^n\pi^{w_p(M)} \\ 
\cdot\sum_{\gamma\in\Gamma'_a} \frac{\bar{\lambda}^{-p^av(\gamma)}}{\big((1-p^a)v(\gamma)\big)!!}\Phi_{v(\gamma),a}\big(\omega(\bar{\lambda}_1),\dots,\omega(\bar{\lambda}_N)\big) \pmod{\pi^{w_p(M)+1}}.
\end{multline}
Equation (7.4) shows that the Hasse invariant of the exponential sum $S(f_{\hat{\lambda}},{\bf e})$ is a sum of mod $p$ solutions of $p$-adically normalized $A$-hypergeometric systems \cite[Equations~(2.2) and~(2.3)]{AS3}.  Theorem~6.6 of \cite{AS3} gives a more detailed version of this result from a different point of view.  

Similar results hold for the corresponding exponential sums over extension fields of ${\mathbb F}_q$.  For $b\in{\mathbb N}_+$ set
\begin{multline}
S_b(f_{\bar{\lambda}},{\bf e}) = \\ 
\sum_{x=(x_1,\dots,x_n)\in({\mathbb F}_{q^b}^\times)^n} \bigg(\prod_{i=1}^n \omega\big({\rm Norm}_{{\mathbb F}_{q^b}/{\mathbb F}_q}(x_i)\big)^{-e_i}\bigg)\Psi\big({\rm Tr}_{{\mathbb F}_{q^b}/{\mathbb F}_p}(f_{\bar{\lambda}}(x))\big),
\end{multline}
an element of ${\mathbb Q}_p(\zeta_p,\zeta_{q-1})$.
Put $W_{ab} = \min_{\gamma\in\Gamma_{ab}}\{abw(R_{\gamma,ab})\}$ and
\[ \Gamma'_{ab} = \{\gamma\in\Gamma_{ab}\mid abw(R_{\gamma,ab}) = W_{ab}\}. \]
Then as in (7.4) one has
\begin{multline}
S_b(f_{\bar{\lambda}},{\bf e})\equiv (-1)^n\pi^{W_{ab}} \\
\cdot\sum_{\gamma\in\Gamma'_{ab}} \frac{\bar{\lambda}^{-p^{ab}v(\gamma)}}{\big((1-p^{ab})v(\gamma)\big)!!}\Phi_{v(\gamma),ab}\big(\omega(\bar{\lambda}_1),\dots,\omega(\bar{\lambda}_N)\big) \pmod{\pi^{W_{ab}+1}}
\end{multline}
where $v(\gamma)\in R_{\gamma,ab}$ is any element such that $w(v(\gamma)) = w(R_{\gamma,ab})$.

By Theorem 6.1, there exists $b$ such that at least one of the $\Phi_{v(\gamma),ab}$ on the right-hand side of (7.6) is the truncation of a $p$-integral $A$-hypergeometric series.  So for any choice of $A$ and ${\bf e}$ there is an extension ${\mathbb F}_{q^b}$ of ${\mathbb F}_q$ such that the study of the given exponential sum over ${\mathbb F}_{q^b}$ involves $p$-integral $A$-hypergeometric series.  These are often starting points for the analysis of the associated $L$-function.  Specifically, the $p$-integral $A$-hypergeometric series that arise in this manner seem to be related to the reciprocal roots of minimal $p$-divisibility of the $L$-function.  It would be nice to make this relationship precise.

\end{document}